\documentclass[11pt,reqno]{amsart}

\usepackage{graphicx}
\usepackage[table]{xcolor}
\usepackage{xspace,tikz}
\usepackage{amsmath,amsthm,amsfonts,amssymb,latexsym,mathrsfs,color,extarrows}
\usepackage{hyperref}

\newtheorem{theorem}{Theorem}
\newtheorem{corollary}[theorem]{Corollary}

\newtheorem{lemma}[theorem]{Lemma}

\newcommand{\lpk}{{\rm lpk\,}}

\newcommand{\pk}{{\rm pk\,}}

\newcommand{\msn}{\mathfrak{S}_n}

\newcommand{\lrf}[1]{\lfloor #1\rfloor}

\newcommand{\arxiv}[1]{\href{http://arxiv.org/abs/#1}{\texttt{arXiv:#1}}}
\title[Central factorial numbers]{Alternating runs of permutations and the central factorial numbers}
\author[Q~Fang]{Qi Fang}
\address{School of Mathematics, Northeastern University, Shenyang 110004, P.R. China}
\email{qifangpapers@stumail.neu.edu.cn (Q~Fang)}
\author[Y.-N~Feng]{Ya-Nan Feng}
\address{School of Mathematics, Northeastern University, Shenyang 110004, P.R. China}
\email{ynfeng@stumail.neu.edu.cn (Y.-N~Feng)}
\author[S.-M.~Ma]{Shi-Mei~Ma}
\address{School of Mathematics and Statistics,
        Northeastern University at Qinhuangdao,
         Hebei 066004, P. R. China}
\email{shimeimapapers@163.com (S.-M. Ma)}
\begin{document}
\maketitle
\begin{abstract}
Let $R(n,k)$ be the number of permutations of $\{1,2,\ldots,n\}$ with $k$ alternating runs.
In this paper, we establish the relationships between $R(n,k)$ and the central factorial numbers of even indices as well as
the number of signed permutations with a given number of alternating runs and the central factorial numbers of odd indices.
The explicit formulas of the peak and left peak polynomials for permutations
and the derivative polynomials of the tangent and secant functions are also established.

\bigskip

\end{abstract}
\date{\today}


\section{Introduction}
Let $\msn$ be the symmetric group of all permutations of $[n]$, where $[n]=\{1,2,\ldots,n\}$.
Let $\pi=\pi(1)\pi(2)\cdots\pi(n)\in\msn$.
We say that $\pi$ changes
direction at position $i$ if either $\pi({i-1})<\pi(i)>\pi(i+1)$, or
$\pi(i-1)>\pi(i)<\pi(i+1)$, where $i\in\{2,3,\ldots,n-1\}$. We say that $\pi$ has $k$ {\it alternating
runs} if there are $k-1$ indices $i$ such that $\pi$ changes
direction at these positions. Let $R(n,k)$ denote the number of
permutations in $\msn$ with $k$ alternating runs.
The enumeration of finite sequences according to the number of alternating runs
has been studied extensively in the literature, starting with Andr\'e~\cite{Andre84} in 1884, restarting with Carlitz~\cite{Carlitz78,Carlitz80,Carlitz81}, and
Comtet~\cite[p.~260-261]{Comtet74} gave an exercise to this topic.
In the past decades, several authors~\cite{CW08,Ma12,Sta08} have considered
the explicit formulas of the numbers $R(n,k)$ as well as their generating functions.
The reader is referred to~\cite{Bona20,Ma2001,Zhuang16,Zhuang17} for the recent progress on this topic.

Following Andr\'e~\cite{Andre84}, the numbers $R(n,k)$ satisfy the recurrence relation
\begin{equation}\label{rnk-rr}
    R(n,k)=kR(n-1,k)+2R(n-1,k-1)+(n-k)R(n-1,k-2)
\end{equation}
for $n,k\geqslant 1$, where $R(1,0)=1$ and $R(1,k)=0$ for $k\geqslant 1$.
Let
$R_n(x)=\sum_{k=1}^{n-1}R(n,k)x^k$. It follows from~\eqref{rnk-rr} that
\begin{equation}\label{Rnx-recu}
R_{n+2}(x)=x(nx+2)R_{n+1}(x)+x\left(1-x^2\right)\frac{\mathrm{d}}{\mathrm{d}x}R_{n+1}(x).
\end{equation}
Below are the polynomials $R_n(x)$'s for $2\leqslant n\leqslant 5$:
\begin{align*}
  R_2(x)& =2x, \\
  R_3(x)& =2x+4x^2, \\
  R_4(x)& =2x+12x^2+10x^3,\\
  R_5(x)& =2x+28x^2+58x^3+32x^4.
\end{align*}

By the method of characteristics,
Carlitz~\cite{Carlitz78} proved that
\begin{equation}\label{CarlitzGF}
\sum_{n=0}^\infty (1-x^2)^{-n/2}\frac{z^n}{n!}\sum_{k=0}^nR(n+1,k)x^{n-k}=\frac{1-x}{1+x}\left(\frac{\sqrt{1-x^2}+\sin z}{x- \cos z}\right)^2.
\end{equation}
Subsequently, Carlitz~\cite[Eq.~(1.3)]{Carlitz80} deduced the following explicit formula of $R(n,k)$:
\begin{equation}\label{Carlitz}
\left\{
  \begin{array}{ll}
 R(2 n-1,2 n-s-2)=\sum_{j=1}^{n}(-1)^{n-j} 2^{-j+2}(2 j-1) ! U({n, j}) M(n,j,s), & \\
 R(2 n, 2 n-s-1)=\sum_{j=1}^{n}(-1)^{n-j} 2^{-j+1}(2 j) ! U({n, j}) M(n,j,s). &
  \end{array}
\right.
\end{equation}
where
$$U(n,j)=\frac{1}{(2j)!} \sum_{i=0}^{2j}(-1)^{i}\binom{2j}{i}(j-i)^{2n},~M(n,j,s)=\sum_{t=0}^{n-j}(-1)^{t}\binom{n-j}{t}\binom{n-2}{s-t}.$$
It should be noted that $U(n,j)$ are central factorial numbers of even indices.
In~\cite{Sta08}, Stanley gave another explict formula:
\begin{equation*}\label{Rnk-Stan}
R(n,k)=\sum_{i=0}^k\frac{1}{2^{i-1}}(-1)^{k-i}z_{k-i}\sum_{\substack{r+2m\leqslant i\\r\equiv i\bmod 2}}(-2)^m\binom{i-m}{(i+r)/2}\binom{n}{m}r^n,
\end{equation*}
where $z_0=2$ and $z_n=4$ for $n\geqslant 1$.
In~\cite{CW08}, Canfield and Wilf pointed out that there is something wrong in~\eqref{Carlitz}, and
showed that
$$R(n,k)=\frac{1}{2^{k-2}}k^n-\frac{1}{2^{k-4}}(k-1)^n+\psi_2(n,k)(k-2)^n+\cdots+\psi_{k-1}(n,k)\quad\text{for $n\geqslant 2$},$$
in which each $\psi_i(n,k)$ is a polynomial in $n$ whose degree in $n$ is $\lrf{i/2}$.
By expressing $R_n(x)$ in terms of the derivative polynomials of tangent function, Ma~\cite{Ma12} found that
\begin{equation*}
R(n,s)=\frac{1}{2^{n-1}}\sum_{k=0}^{\lfloor(n+1)/2\rfloor}p(n,n-2k+1)E(n,k,s),
\end{equation*}
where $$p(n,n-2k+1)=(-1)^{k}\sum_{i\geqslant 1}i!{n \brace i}(-2)^{n-i}\left\{\binom{i}{n-2k}-\binom{i}{n-2k+1}\right\},$$
$$E(n,k,s)=\sum_{j=0}^{\min (k,s)}(-1)^{k-j}\binom{n-k-1}{s-j}\binom{k}{j}.$$

The peak polynomials for permutations and derivative polynomials of
trigonometric functions appeared frequently in algebra, number theory, geometry and combinatorics,
see~\cite{Hoffman99,Ma123,Petersen07,Zhuang16,Zhuang17} and references therein.
In the next section, we first present a revision of~\eqref{Carlitz},
and then we provide explicit formulas for the peak and left peak polynomials as well as the derivative polynomials of the tangent and secant functions.
An explicit formula of the alternating run polynomials of signed permutations is also established.
From the results of this paper, one can see that sometimes enumerative polynomials naturally occur in pairs, and it is often much better to tackle both at the same time.
\section{Explicit formula in terms of central factorial numbers}\label{sec02}
The {\it central factorial numbers of the second kind} $T(n,k)$ are defined in Riordan's book~\cite[p.~213-217]{Riordan68}
by $$x^n=\sum_{k=0}^nT(n,k)x\prod_{i=1}^{k-1}\left(x+\frac{k}{2}-i\right).$$
As usual, we denote by $U(n,k)=T(2n,2k)$ and $V(n,k)=4^{n-k}T(2n+1,2k+1)$ for all $n,k\geqslant 0$.
By definitions, these numbers satisfy the recurrence relations
$$U(n,k)=U(n-1,k-1)+k^2U(n-1,k),$$
\begin{equation}\label{Vnk-recu}
V(n,k)=V(n-1,k-1)+(2k+1)^2V(n-1,k),
\end{equation}
with the initial conditions $U(1,1)=1$,~$U(1,k)=0$ if $k\neq1$, $V(0,0)=1$ and $V(0,k)=0$ if $k\neq0$, see~\cite{Garvan11,Zeng10,Kim22} for details.
The {\it central factorial numbers of even indices} $U(n,k)$ first appeared in a paper of MacMahon~\cite[p.~106]{MacMahon}, and they
can be defined by
$$x^n=\sum_{k=1}^nU(n,k)\prod_{i=1}^{k}\left(x-(i-1)^2\right).$$
The numbers $U(n,k)$ count partitions of the set $\{1,-1,2,-2,\ldots,n,-n\}$ with $k$ blocks $B_1,\ldots,B_k$ such that for each $j\in [k]$, if
$i$ is the least integer such that $i$ or $-i$ belongs to $B_j$, then $\{i,-i\}$ is a subset of $B_j$, see~\cite[Section~2.1]{Zeng10} for details.
According to~\cite[p.~214]{Riordan68}, we have
$$\sum_{n=0}^\infty\sum_{k=0}^nV(n,k)x^{2k+1}\frac{z^{2n+1}}{(2n+1)!}=\sinh(x\sinh(z))=xz+(x+x^3)\frac{z^3}{3!}+(x+10x^3+x^5)\frac{z^5}{5!}+\cdots.$$
Gelineau and Zeng~\cite[Theorem~8]{Zeng10} found that {\it the central factorial numbers of odd indices} $V(n,k)$ count
partitions of $[2n+1]$ into $2k+1$ blocks of odd cardinality. Explicitly, one has
$$V(n,k)=\frac{1}{(2k)!4^k}\sum_{m=0}^k(-1)^{k-m}\frac{(2m+1)^{2n+1}}{k+m+1}\binom{2k}{k+m}.$$
We refer the reader to~\cite[A036969,A160562]{Sloane} for various results on central factorial numbers.

As pointed out in~\cite{Kim22}, even though the central factorial numbers are less
known than Stirling numbers, they are as important as Stirling numbers.
In this paper, we give some applications of the central factorial numbers in the study of
the alternating run polynomials, peak and left peak polynomials as well as the derivative polynomials of tangent and secant functions.
\subsection{Main results}
\hspace*{\parindent}

The first main result of this paper is given as follows.
\begin{theorem}\label{Mainthm01}
For any $n\geqslant 1$, one has
\begin{equation}\label{Carlitz01}
\begin{split}
\frac{R_{2n-1}(x)}{(1+x)^{n-2}}&=\sum_{j=1}^{n}{2^{-j+2}(2j-1)!U(n,j)x^j(1-x)^{n-j}},  \\
\frac{R_{2n}(x)}{(1+x)^{n-1}}&=\sum_{j=1}^{n}{2^{-j+1}(2j)!U(n,j)x^j(1-x)^{n-j}}.
\end{split}
\end{equation}
Equivalently, we have
\begin{equation*}
\begin{split}
\sum_{s=0}^{2n-2} R(2n-1, s) x ^{2n-2-s}&=\sum_{j=1}^{n}{2^{-j+2}(2j-1)!U(n,j)(x-1)^{n-j}(x+1)^{n-2}},  \\
\sum_{s=0}^{2n-1} R(2n, s) x ^{2n-1-s}&=\sum_{j=1}^{n}{2^{-j+1}(2j)!U(n,j)(x-1)^{n-j}(x+1)^{n-1}}.
\end{split}
\end{equation*}
So we get
\begin{equation}\label{Carlitz02}
\left\{
  \begin{array}{ll}
 R(2 n-1,2 n-s-2)=\sum_{j=1}^{n}(-1)^{n-j} 2^{-j+2}(2 j-1) ! U({n, j}) M(n,j,s), & \\
 R(2 n, 2 n-s-1)=\sum_{j=1}^{n}(-1)^{n-j} 2^{-j+1}(2 j) ! U({n, j}) N(n,j,s), &
  \end{array}
\right.
\end{equation}
where $M(n,j,s)$ are the coefficients of $x^s$ in $(1+x)^{n-2}(1-x)^{n-j}$ and $N(n,j,s)$ are the coefficients of $x^s$ in $(1+x)^{n-1}(1-x)^{n-j}$.
\end{theorem}

The proof of Theorem~\ref{Mainthm01} will be given in subsection~\ref{proof}.
The following result is immediate.
\begin{corollary}[\cite{BE00}]\label{Bona}
The polynomial $R_n(x)$ is divisible by $(x+1)^{\lfloor n/2\rfloor-1}$ for any $n\geqslant 2$.
\end{corollary}
Recently, B\'ona~\cite{Bona20} gave a proof of Corollary~\ref{Bona} by introducing a group action on permutations.

Let $\pi=\pi(1)\pi(2)\cdots\pi(n)\in\msn$.
An {\it interior peak} of $\pi$ is an index $i\in\{2,3,\ldots,n-1\}$ such that $\pi(i-1)<\pi(i)>\pi(i+1)$. Let $\pk(\pi)$ denote the number of interior peaks of $\pi$.
The classical {\it peak polynomials} are defined by
\begin{equation*}
P_n(x)=\sum_{\pi\in\msn}x^{\pk(\pi)},
\end{equation*}
which have been extensively studied in the past decades, see~\cite{Ma123,Petersen07,Zhuang17} and references therein.
There is a close connection between $R_n(x)$ and $P_n(x)$ (see~\cite[Corollary 2]{Ma123}):
\begin{equation}\label{Rnx-Pnx}
R_n(x)=\frac{x(1+x)^{n-2}}{2^{n-2}}P_n\left(\frac{2x}{1+x}\right)~\text{for $n\geqslant 2$}.
\end{equation}
Combining~\eqref{Carlitz01} and~\eqref{Rnx-Pnx}, we get the following result.
\begin{corollary}\label{Carlitz022}
For any $n\geqslant 1$, we have
\begin{equation}\label{Carlitz02}
\begin{split}
xP_{2n-1}(x)&=\sum_{j=1}^{n}{2^{2n-2j}(2j-1)!U(n,j)x^j(1-x)^{n-j}},  \\
xP_{2n}(x)&=\sum_{j=1}^{n}{2^{2n-2j}(2j)!U(n,j)x^j(1-x)^{n-j}}.
\end{split}
\end{equation}
\end{corollary}

Let $\pi=\pi(1)\pi(2)\cdots\pi(n)\in\msn$.
A {\it left peak} of $\pi$ is an index $i\in[n-1]$ such that $\pi(i-1)<\pi(i)>\pi(i+1)$,
where we take $\pi(0)=0$. Let $\lpk(\pi)$ denote {\it the number of left peaks} of $\pi$.
The {\it left peak polynomials} are defined by
$$ {\widehat{P}}_n(x)=\sum_{\pi\in\msn}x^{\lpk(\pi)}.$$
It is well known that the polynomials $\widehat{P}_n(x)$ satisfy the following recurrence relation
\begin{equation}\label{Wnxleftpeak-recu}
\widehat{P}_{n+1}(x)=(nx+1)\widehat{P}_{n}(x)+2x(1-x)\frac{\mathrm{d}}{\mathrm{d}x}\widehat{P}_{n}(x),
\end{equation}
with initial conditions $\widehat{P}_{1}(x)=1$ and $\widehat{P}_{2}(x)=1+x$, see~\cite[A008971]{Sloane} and references therein.
As a dual of Corollary~\ref{Carlitz022}, we can now present the second main result of this paper.
\begin{theorem}\label{thm02}
For any $n\geqslant 1$, we have
\begin{equation}\label{PP}
\begin{split}
\widehat{P}_{2n}(x)&=\sum_{j=0}^{n}{(2j)!V(n,j)x^j(1-x)^{n-j}},\\
\widehat{P}_{2n+1}(x)&=\sum_{j=0}^{n}{(2j+1)!V(n,j)x^j(1-x)^{n-j}}.
\end{split}
\end{equation}
\end{theorem}
\begin{proof}
Note that $$\widehat{P}_{2}(x)=1+x=(1-x)+2x,~\widehat{P}_{3}(x)=1+5x=(1-x)+6x,$$
$$\widehat{P}_{4}(x)=1+18x+5x^2=(1-x)^2+20x(1-x)+24x^2,$$
$$\widehat{P}_{5}(x)=1+58x+61x^2=(1-x)^2+60x(1-x)+120x^2.$$
Thus the two expressions in~\eqref{PP} hold for $n=1,2$. We proceed by induction. For $m\geqslant 2$,
assume that the result holds for $n=m$. It follows from~\eqref{Wnxleftpeak-recu} that
\begin{equation*}
\begin{split}
\widehat{P}_{2m+2}(x)&=(2mx+x+1)\widehat{P}_{2m+1}(x)+2x(1-x)\frac{\mathrm{d}}{\mathrm{d}x}\widehat{P}_{2m+1}(x)\\
&=\left((2mx+1)+x\right)\sum_{j=0}^{m}{(2j+1)!V(m,j)x^j(1-x)^{m-j}}+\\
&2\sum_{j=0}^{m}(2j+1)!jV(m,j)x^j(1-x)^{m-j+1}-2\sum_{j=0}^{m}(2j+1)!(m-j)V(m,j)x^{j+1}(1-x)^{m-j}\\
&=\sum_{j=0}^{m}{(2j+1)!V(m,j)x^{j+1}(1-x)^{m-j}}+\sum_{j=0}^{m}{(2j+1)!(2j+1)V(m,j)x^j(1-x)^{m-j}},
\end{split}
\end{equation*}
where the last equality follows from the fact that $2mx+1+2j(1-x)-2(m-j)x=2j+1$.
Using $1=(1-x)+x$, we obtain
\begin{equation*}
\begin{split}
\widehat{P}_{2m+2}(x)&=\sum_{j=0}^{m}{(2j+2)!V(m,j)x^{j+1}(1-x)^{m-j}}+\sum_{j=0}^{m}{(2j+1)!(2j+1)V(m,j)x^{j}(1-x)^{m-j+1}}.
\end{split}
\end{equation*}
It follows from~\eqref{Vnk-recu} that the
coefficient of $x^j(1-x)^{m-j+1}$ on the right-hand side is $(2j)!V(m+1,j)$.
By the recurrence relation~\eqref{Wnxleftpeak-recu}, we have
\begin{equation*}
\begin{split}
\widehat{P}_{2m+3}(x)&=(2mx+2x+1)\widehat{P}_{2m+2}(x)+2x(1-x)\frac{\mathrm{d}}{\mathrm{d}x}\widehat{P}_{2m+2}(x)\\
&=(2mx+2x+1)\sum_{j=0}^{m+1}{(2j)!V(m+1,j)x^j(1-x)^{m+1-j}}+\\
&2(1-x)\sum_{j=0}^{m+1}{(2j)!jV(m+1,j)x^j(1-x)^{m+1-j}}-\\
&2x\sum_{j=0}^{m+1}{(2j)!(m+1-j)V(m+1,j)x^j(1-x)^{m+1-j}}.
\end{split}
\end{equation*}
The coefficient of $x^j(1-x)^{m+1-j}$ on the right-hand side is $(2j+1)!V(m+1,j)$, since
$$2mx+2x+1+2j(1-x)-2x(m+1-j)=2j+1.$$
Therefore, the two expressions in~\eqref{PP} hold for $n=m+1$. This completes the proof.
\end{proof}

Denote by $B_n$ the hyperoctahedral group of rank $n$. Elements $\pi$ of $B_n$ are signed permutations of the set $\pm[n]$ such that $\pi(-i)=-\pi(i)$ for all $i$, where $\pm[n]=\{\pm1,\pm2,\ldots,\pm n\}$. As usual, we identify a signed permutation $\pi=\pi(1)\cdots\pi(n)$ with the word $\pi(0)\pi(1)\cdots\pi(n)$, where $\pi(0)=0$.
A {\it run} of a signed permutation $\pi$ is defined as a maximal interval of consecutive elements on which the elements of $\pi$ are monotonic in the order $$\cdots<\overline{2}<\overline{1}<0<1<2<\cdots.$$
The {\it up signed permutations} are signed permutations with $\pi(1)> 0$.
Let $\widehat{R}(n,k)$ denote the number of up signed permutations in $B_n$ with $k$ alternating runs and let $\widehat{R}_n(x)=\sum_{k=1}^n\widehat{R}(n,k)x^k$.
Following~\cite[Theorem~4]{Chow14}, we have
\begin{equation*}\label{Tnx-Pnx}
\widehat{R}_n(x)=x(1+x)^{n-1}\widehat{P}_n\left(\frac{2x}{1+x}\right)
\end{equation*}
for $n\geqslant 1$. Combining this with Theorem~\ref{thm02}, we obtain the following result.
\begin{corollary}
For $n\geqslant 1$, we have
\begin{equation*}
\begin{split}
\widehat{R}_{2n}(x)&=x(1+x)^{n-1}\sum_{j=0}^{n}{2^j(2j)!V(n,j)x^j(1-x)^{n-j}},\\
\widehat{R}_{2n+1}(x)&=x(1+x)^n\sum_{j=0}^{n}2^j{(2j+1)!V(n,j)x^j(1-x)^{n-j}}.
\end{split}
\end{equation*}
\end{corollary}

In the sequel, we discuss derivative polynomials of the tangent and secant functions.
In 1879, Andr\'e~\cite{Andre79} observed that
\begin{equation*}
 \sum_{n=0}^\infty E_n\frac{z^n}{n!}=\tan z+\sec z=1+z+\frac{z^2}{2!}+2\frac{z^3}{3!}+5\frac{z^4}{4!}+16\frac{z^5}{5!}+\cdots.
\end{equation*}
Note that
$$\sum_{n=0}^\infty E_{2n+1}\frac{z^{2n+1}}{(2n+1)!}=\tan z,~~\sum_{n=0}^\infty E_{2n}\frac{z^{2n}}{(2n)!}=\sec z.$$
For this reason the numbers $E_{2n+1}$ are sometimes called {\it tangent numbers} and the numbers $E_{2n}$ are {\it secant numbers}.
The {\it derivative polynomials}
of tangent and secant functions are respectively defined as follows:
$$\frac{\mathrm{d}^n}{\mathrm{d}\theta^n}\tan\theta=Q_n(\tan \theta),~\frac{\mathrm{d}^n}{\mathrm{d}\theta^n}\sec\theta=\sec\theta \cdot \widehat{Q}_n(\tan \theta).$$
The study of these polynomials was initiated by Knuth and Buckholtz~\cite{Knuth}.
They noted that $Q_{2n-1}(0)=E_{2n-1}$ and $\widehat{Q}_{2n}(0)=E_{2n}$.
It is well known that (see~\cite{Carlitz78,Hoffman95}):
\begin{equation*}
Q(x;z)=\sum_{n=0}^\infty Q_n(x)\frac{z^n}{n!}=\frac{x+\tan z}{1-x\tan z},
\end{equation*}
\begin{equation*}
\widehat{Q}(x;z)=\sum_{n=0}^\infty \widehat{Q}_n(x)\frac{z^n}{n!}=\frac{\sec z}{1-x\tan z}.
\end{equation*}
By the chain rule, we get $$Q_{n+1}(x)=(1+x^2)\frac{\mathrm{d}}{\mathrm{d}x}Q_n(x),~
\widehat{Q}_{n+1}(x)=(1+x^2)\frac{\mathrm{d}}{\mathrm{d}x}\widehat{Q}_n(x)+x\widehat{Q}_n(x),$$
with $Q_0(x)=x$ and $\widehat{Q}_0(x)=1$.
According to~\cite[Theorem~2]{Ma121}, for $n\geqslant 1$, we have
\begin{equation*}\label{PQ}
Q_n(x)=(x^{n-1}+x^{n+1})P_n(1+x^{-2}),~\widehat{Q}_n(x)=x^n\widehat{P}_n(1+x^{-2}).
\end{equation*}
Combining this with Corollary~\ref{Carlitz022} and Theorem~\ref{thm02}, we obtain the following result.
\begin{theorem}
For $n\geqslant 1$, we have
\begin{equation*}
\begin{split}
Q_{2n-1}(x)&=\sum_{j=1}^{n}{(-4)^{n-j}}(2j-1)!U(n,j)(1+x^2)^j,  \\
Q_{2n}(x)&=x\sum_{j=1}^{n}{(-4)^{n-j}}(2j)!U(n,j)(1+x^2)^j,\\
\widehat{Q}_{2n}(x)&=\sum_{j=0}^{n}(-1)^{n-j}{(2j)!}V(n,j)(1+x^2)^j,\\
\widehat{Q}_{2n+1}(x)&=x\sum_{j=0}^{n}(-1)^{n-j}{(2j+1)!}V(n,j)(1+x^2)^j.
\end{split}
\end{equation*}
\end{theorem}

There has been much work on special values and explicit formulas of the derivative polynomials of trigonometric functions,
see~\cite{Cvijovic09,Hoffman99} for instance.
In~\cite{Hoffman99}, Hoffman noted that $\widehat{Q}_n(1)$ are the Springer numbers of root systems of type $B_n$, which also count snakes of type $B_n$.
A {\it snake} of type $B_n$ is a signed permutation $\pi(1)\pi(2)\cdots\pi(n)$ of $B_n$ such that $0<\pi(1)>\pi(2)<\cdots\pi(n)$.
Setting $s_n=\widehat{Q}_n(1)$, then we have
$$\sum_{n=0}^{\infty}s_n\frac{z^n}{n!}=\frac{1}{\cos z-\sin z}.$$
From the above discussion, we get the following result.
\begin{corollary}
For $n\geqslant 1$, we have
\begin{equation*}
\begin{split}
E_{2n-1}&=\sum_{j=1}^{n}{(-4)^{n-j}}(2j-1)!U(n,j),~~
E_{2n}=\sum_{j=0}^{n}(-1)^{n-j}{(2j)!}V(n,j),\\
s_{2n}&=\sum_{j=0}^{n}(-1)^{n-j}{(2j)!}2^jV(n,j),~~s_{2n+1}=\sum_{j=0}^{n}(-1)^{n-j}{(2j+1)!}2^jV(n,j).
\end{split}
\end{equation*}
\end{corollary}
\subsection{The proof of Theorem~\ref{Mainthm01}}\label{proof}
\hspace*{\parindent}

Along the same lines as in~\cite{Carlitz80}, setting $x=\cos 2\theta$ and replacing $z$ by $2z$ in~\eqref{CarlitzGF}, we get
\begin{equation*}
 \sum_{n=0}^\infty (\sin 2\theta)^{-n}2^n\frac{z^n}{n!}\sum_{k=0}^nR(n+1,k){\cos^{n-k} 2\theta}=
   \tan^2 \theta \cot^2(\theta-z).
\end{equation*}
Thus by replacing $z$ by $-z$, we obtain
\begin{align*}
&\sum_{n=0}^\infty (\sin 2\theta)^{-n}2^n(-1)^n \frac{z^n}{n!}\sum_{k=0}^nR(n+1,k){\cos^{n-k} 2\theta}\\
&=\tan^2 \theta \cot^2(\theta+z)\\
&=\tan^2 \theta \csc^2 (\theta+z) -\tan^2\theta.
\end{align*}
By the Taylor's theorem, it is easy to derive that
\begin{equation}\label{calitzcsc}
(-1)^{n} 2^{n}(\sin 2 \theta)^{-n} \sum_{k=0}^{n} R(n+1, k)\cos^{n-k} 2 \theta=\tan ^{2} \theta \frac{\mathrm{d}^{n}}{\mathrm{d}\theta^{n}} \csc ^{2} \theta
\end{equation}
Carlitz deduced the following result by a simple induction.
\begin{lemma}[{\cite[Eqs~(2.11),~(2.12)]{Carlitz80}}]\label{Lemma80}
For $n\geqslant 1$, one has
\begin{equation*}
\frac{{\mathrm{d}}^{2n-2}}{{\mathrm{d}}\theta^{2n-2}}\csc^{2}\theta=\sum_{j=1}^{n}{(-1)^{n-j}2^{2n-2j}(2j-1)!U({n,j})\csc^{2j}\theta},
\end{equation*}
\begin{equation*}
\frac{{\mathrm{d}}^{2n-1}}{{\mathrm{d}}\theta^{2n-1}}\csc^{2}\theta=\sum_{j=1}^{n}{(-1)^{n-j+1}2^{2n-2j}(2j)!U({n,j})\csc^{2j}\theta}\cot\theta,
\end{equation*}
\end{lemma}

\noindent{\bf A proof
Theorem~\ref{Mainthm01}:}
The proof of the expression of $\sum_{s=0}^{2n-2} R(2n-1, s) x ^{2n-2-s}$ is the same as in~\cite{Carlitz80}. However, we give a proof of it for completeness.
Replacing $n$ by $2n-2$, then~\eqref{calitzcsc} becomes
$$2^{2n-2}(\sin 2 \theta)^{-2n+2} \sum_{s=0}^{2n-2} R(2n-1, s) \cos ^{2n-2-s} 2\theta=\tan ^{2} \theta \frac{\mathrm{d}^{2n-2}}{\mathrm{d}\theta^{2n-2}} \csc ^{2} \theta.$$
Combining Lemma~\ref{Lemma80} and the double angle formula $\sin 2\theta=2\sin \theta\cos \theta$, we get
\begin{equation*}
\begin{split}
&\sum_{s=0}^{2n-2} R(2n-1, s) \cos ^{2n-2-s} 2\theta\\
&=2^{2-2n}(\sin 2 \theta)^{2n-2}\tan ^{2} \theta \frac{\mathrm{d}^{2n-2}}{\mathrm{d} \theta^{n}} \csc ^{2} \theta\\
&=\sum_{j=1}^{n}{(-1)^{n-j}2^{2n-2j}(2j-1)!U({n,j})\sin ^{2n-2j}\theta \cos^{2n-4}\theta}.
\end{split}
\end{equation*}
Since $\sin^2\theta=\frac{1-\cos 2\theta}{2},~\cos^2\theta=\frac{1+\cos 2\theta}{2}$,
we obtain
\begin{equation*}
\begin{split}
&\sum_{s=0}^{2n-2} R(2n-1, s) \cos^{2n-2-s} 2\theta\\
&=\sum_{j=1}^{n}{(-1)^{n-j}2^{2-j}(2j-1)!U({n,j})}(1-\cos 2\theta)^{n-j}(1+\cos 2\theta)^{n-2}.
\end{split}
\end{equation*}
Setting $x=\cos 2\theta$, we get the expression of $\sum_{s=0}^{2n-2} R(2n-1, s) x ^{2n-2-s}$.

Similarly, replacing $n$ by $2n-1$, the identity~\eqref{calitzcsc} becomes
$$\sum_{s=0}^{2n-1} R(2n, s) \cos ^{2n-1-s} 2 \theta=(-1)2^{1-2n}(\sin 2 \theta)^{2n-1}\tan ^{2} \theta \frac{\mathrm{d}^{2n-1}}{\mathrm{d} \theta^{2n-1}} \csc ^{2} \theta.$$
Then combining Lemma~\ref{Lemma80} and double angle formulas, we get
\begin{equation*}
\begin{split}
&\sum_{s=0}^{2n-1} R(2n, s) \cos ^{2n-1-s} 2 \theta\\
&=\sum_{j=1}^{n}{(-1)^{n-j}2^{2n-2j}(2j)!U({n,j})\sin ^{2n-2j}\theta\cos^{2n-2} \theta}\\
&=\sum_{j=1}^{n}{(-1)^{n-j}2^{1-j}(2j)!U({n,j})(1-\cos 2\theta)^{n-j}(1+\cos 2\theta)^{n-1}}
\end{split}
\end{equation*}
Setting $x=\cos 2\theta$, we get the expression of $\sum_{s=0}^{2n-1} R(2n, s) x ^{2n-1-s}$. This completes the proof. \qed
\section*{Acknowledgements.}
This work is supported by NSFC (12071063).

\end{document}